\documentclass[reqno]{amsart}
\usepackage{amsaddr}

\usepackage{amsmath,amssymb,amsthm,esint,dsfont,xcolor,framed}
\usepackage[english]{babel}
\usepackage{soul}
\usepackage[title]{appendix}
\usepackage{hyperref} 
\usepackage{comment}
\usepackage[margin=3cm]{geometry}
\numberwithin{equation}{section}

\newcommand{\abs}[1]{{\vert #1\vert}}

\newcommand{\R}{{\mathbb R}}
\newcommand{\RR}{\mathbb{R}}

\DeclareMathOperator{\dist}{dist}

\newtheorem{thm}{Theorem}[section]
\newtheorem{lem}[thm]{Lemma}

\newtheorem{cor}[thm]{Corollary}

\theoremstyle{definition}

\newtheorem{rem}[thm]{Remark}

\title[Electrostatics Analogy in Nematic Suspensions]{Far-Field Expansions for Harmonic Maps and the Electrostatics Analogy in Nematic Suspensions}
\date{\today}
\author[Alama,Bronsard]{Stan Alama \and Lia Bronsard}
\address{Department of Mathematics and Statistics, McMaster University, Hamilton, ON, Canada.}
\email{alama@mcmaster.ca,bronsard@mcmaster.ca}
\author[Lamy]{Xavier Lamy}
\address{Institut de Math\'ematiques de Toulouse; UMR 5219, Universit\'e de Toulouse; CNRS, UPS IMT, F-31062 Toulouse Cedex 9, France}
\email{Xavier.Lamy@math.univtoulouse.fr}
\author[Venkatraman]{Raghavendra Venkatraman}
\address{Courant Institute of Mathematical Sciences, New York University, NY, USA.}
\email{raghav@cims.nyu.edu}

\begin{document}

	\begin{abstract}
		For a smooth bounded domain $G\subset\R^3$  we consider maps $n\colon\mathbb R^3\setminus G\to\mathbb S^2$ minimizing the energy \(E(n)=\int_{\mathbb R^3\setminus G}|\nabla n|^2 +F_s(n_{\lfloor\partial G})\) among $\mathbb S^2$-valued map such that $n(x)\approx n_0$ as $|x|\to\infty$. 
		This is a model for a particle $G$ immersed in nematic liquid crystal. The surface energy $F_s$ describes the anchoring properties of the particle, and can be quite general. 
		We prove that such minimizing map $n$ has an asymptotic expansion in powers of $1/r$.
		Further, we show that the leading order $1/r$ term is uniquely determined by the far-field condition $n_0$ for almost all $n_0\in\mathbb S^2$, 
		by relating it to the gradient of the minimal energy with respect to $n_0$. 
		We derive various consequences of this relation in physically motivated situations:
		when the orientation of the particle $G$ is stable relative to a prescribed far-field alignment $n_0$; and when the particle $G$ has some rotational symmetries.
		In particular, these corollaries justify some approximations that can be found in the physics literature to describe  nematic suspensions via 
		a so-called electrostatics analogy.
	\end{abstract}
	
	\maketitle
	
	\section{Introduction}

	The goal of this work is to investigate the so-called electrostatics analogy in the analysis of nematic suspensions or colloids: these consist of small particles immersed in a  nematic liquid crystal matrix. 
	The presence of these particles and their alignment  induces elastic strains in nematic medium; 
	what results is a complex strain-alignment coupling yielding novel high-functional composite materials. 
	Examples include dilute ferronematics, where the suspended particles are ferromagnetic inclusions; organizing carbon nanotubes using liquid crystals; ferroelectrics; and living liquid crystals, where the suspended particles are swimming bodies (e.g. flagellated bacteria). Further details on the numerous applications of such systems may be found in  the review articles \cite{lavrentovich20,musevic19}.

	Mathematical studies of colloid inclusions in nematics have tended to follow two different directions.  Several papers have addressed homogenization of nematics with a dense array of colloids (see, e.g., \cite{BCG05,BK,CDGP, CZ20,CZ20design}), while others consider the presence of point or ring singularities induced by a single colloid particle (see, e.g., \cite{alamabronsardlamy16saturn,alamabronsardlamy17, ACS21, ABGL21, ACS22}). In this paper we adopt the setting of the second set of papers, but concentrate on the effect of the colloid geometry on the far-field behavior of the nematic rather than the local structure of singularities near the colloid surface.

	The  electrostatics analogy is {commonly used} to describe colloidal suspensions in the case of a dilute concentration of particles. It originates in the work \cite{brocharddegennes70} by Brochard and de Gennes, and has been developed further by several authors in the physics literature \cite{kuksenok96, ramaswamy96,lubensky98}. It relies on considering   each single  particle separately and postulating that:
	\begin{itemize}
		\item far away from the particle the distortion in nematic alignment can be viewed as a perturbation of uniform alignment and taken to solve the corresponding linearized equation -- the representation formula for solutions of that linearized equation then  provides a specific asymptotic expansion,
		\item the first few coefficients of that asymptotic expansion are characterized by the properties (size, symmetries, etc.) of the   particle.
	\end{itemize}
	Then one formally replaces the nonlinear effect of each colloid particle by some singular source terms (derivatives of Dirac masses) in the linearized equation,
	according to the terms in the asymptotic expansion, which are derivatives of the fundamental solution (see Remark \ref{r:eqn0}). 
	In the one-constant approximation for the  elastic energy of the nematic, this amounts to the equation satisfied by an electric potential in the presence of charged multipoles, hence the name ``electrostatics analogy''.
	This simplification, intuitively valid for dilute enough suspensions, allows for
	an explicit  calculation of the energy of a given configuration in terms of the respective positions and properties of each   particle, leading to the ultimate goal: computation of interparticle interactions.
	
	In this article we provide a few elements towards mathematically quantifying the  electrostatics analogy, rigorously obtaining an asymptotic expansion for solutions of the original non-linear and non-convex minimization problem, and comparing it with a multipole expansion of a harmonic function.
	What seems to us the most challenging part is the second bullet-point above: 
	relating the coefficients of the asymptotic expansion to the particle's properties.
	Indeed, various mathematical obstacles defy a straightforward calculation of an expansion of minimizers:  for instance, minimizers may not be unique, and it is unknown whether the symmetry of the particle system imposes a corresponding symmetry on the minimizing nematic configuration. 
	Nevertheless we do obtain some results in that direction for the leading-order term of the expansion.

	Specifically, we consider a single   particle  $G\subset\R^3$ (smooth and bounded) surrounded by nematic liquid crystal. 
	A configuration of nematic alignment is represented by a director field $n\colon\R^3\setminus G\to \mathbb S^2$, and its energy (within the one-constant approximation) is given by
	\begin{align*}
		E(n)=\int_{\R^3\setminus G}|\nabla n|^2 + F_s(n_{\lfloor\partial G}),
	\end{align*}
	where $F_s\colon H^{1/2}(\partial G;\mathbb S^2)\to [0,\infty]$ can be a very general surface energy reflecting the particle's anchoring properties. 
	Uniform alignment
	at far field,
	loosely expressed as  $n(x)\approx n_0\in\mathbb S^2$ for $r=|x|\to\infty$, 
	is imposed through the condition
	\begin{align*}
		\int_{\R^3\setminus G} \frac{|n-n_0|^2}{1+r^2}  \lesssim \int_{\R^3\setminus G}|\nabla n|^2   <\infty.
	\end{align*}
	
	In other words, we are imposing that $n-n_0$ belongs to the completion of smooth maps with bounded support, with respect to the distance induced by the $H^1$ semi-norm; the weight $1/(1+r^2)$ is given by Hardy's inequality.
	Here and in the rest of the article, $A\lesssim B$ means $A\leq C B$ for some absolute constant $C>0$
	
	Equilibrium configurations satisfy the harmonic map equation
	\begin{align*}
		-\Delta n =|\nabla n|^2 n\qquad\text{in }\R^3\setminus G.
	\end{align*}
	Loosely speaking, we prove that:
	\begin{itemize}
		\item minimizing configurations have an asymptotic expansion determined by the linearized equation $\Delta n=0$, however one cannot   discard non-harmonic corrections -- see Theorem~\ref{t:expansion};
		\item generically, the leading-order $\mathcal O(1/r)$ term in that expansion is uniquely determined by the particle $G$  and the far-field uniform alignment $n_0$ -- see Theorem~\ref{t:torque}.
	\end{itemize}
	The first point
	is a result about minimizing harmonic maps in an exterior domain, 
	independent of the presence of a particle (since we do not explicitly relate the expansion's coefficients to the particle).
	The second point is obtained by connecting the leading-order term to the 
	variation of minimal energy induced by keeping the particle $G$ fixed and rotating the far-field alignment $n_0$. 
	This is related to formal calculations in \cite{brocharddegennes70}
	for the torque exerted by the particle on the nematic (see Remark~\ref{r:BdG}).
	
	We have not been able yet to obtain similar characterizations for the next-order terms in the expansion.

	In terms of the electrostatics analogy developed in the physics literature, 
	the main input of our results is to
	clarify the first postulate 
	(that the far field distortions generated by a particle are purely harmonic to large order)
	by sheding new light on the second postulate (that these distortions are uniquely characterized by the particle).
	More precisely, in \cite{brocharddegennes70,kuksenok96,ramaswamy96,lubensky98}, 
	the possible presence of nonharmonic corrections is either not considered, 
	or implicitly deduced from a hypothetical uniqueness principle which would ensure
	that symmetry properties of the particle directly translate into symmetry properties of the full configuration (such uniqueness/symmetry principle seems however difficult to prove).
	Here instead we deduce that nonharmonic corrections are negligible 
	from our characterization of the leading-order term, bypassing any uniqueness or symmetry properties of the full configuration.
	This is valid for instance in the case of a spherical particle (see Corollary~\ref{c:sym}), but also when the orientation of the particle is at equilibrium (locally minimizing relative to variations in the prescribed far-field alignment, see Remark~\ref{r:eqn0}), independently of its symmetry properties. 
	Moreover we stress that, for an axisymmetric particle, it is not evident that the equilibrium orientation should be the most symmetric one (see Remark~\ref{r:axisym}).

	Below we state our results in more detail.

	\subsection{Far-field expansion for harmonic maps}
	
	Our
	first main result is a far-field
	expansion for
	harmonic maps in an exterior domain, which (by rescaling) we may without loss of generality assume to contain $\mathbb R^3\setminus \overline B_1$. 
	Our first main result is a far-field expansion for such minimizing maps.

	\begin{thm}\label{t:expansion}
		Let $n_0\in\mathbb S^2$. 
		Assume that $n\in H^1_{loc}(\R^3\setminus \overline B_1;\mathbb S^2)$
		satisfies
		\begin{align}\label{eq:farfield}
			\int_{\R^3\setminus \overline B_1} \frac{|n-n_0|^2}{r^2} \lesssim\int_{\R^3\setminus \overline B_1}|\nabla n|^2 <\infty,
		\end{align}
		and $n$ is locally energy-minimizing, that is,
		\begin{align*}
			\int_{\R^3\setminus\overline B_1}|\nabla n|^2 \leq \int_{\R^3\setminus \overline B_1}|\nabla \tilde n|^2,
		\end{align*}
		for any $\mathbb S^2$-valued map $\tilde n$ which agrees with $n$ outside of a compact subset of $\mathbb R^3\setminus \overline B_1$.
		Then 
		there exist $v_0,p_j,c_{k\ell}\in\R^3$ ($1\leq j,k,\ell\leq 3$) 
		such that, 
		as $r=|x|\to\infty$,
		\begin{align}\label{eq:expansion}
			n&= n_0 + n_{harm} +n_{corr} + \mathcal O\left(\frac{1}{r^4}\right),\\
			n_{harm}& =\frac{1}{r}v_0 +  \sum_{j=1}^3 p_j\partial_j\left(\frac 1r\right) 
			+ \sum_{k,\ell=1}^3 c_{k\ell}\partial_k\partial_\ell\left(\frac 1r\right),\quad v_0,p_j,c_{k\ell}\in\R^3,\nonumber\\
			n_{corr}& = -\frac{|v_0|^2}{r^2}n_0 - \frac{|v_0|^2}{6r^3}v_0
			-\frac{1}{3r}   \sum_{j=1}^3 v_0\cdot p_j\,\partial_j\left(\frac 1r\right) \, n_0.
			\nonumber
		\end{align}
		Moreover the vectors $v_0$, $p_j$ ($j=1,2,3$) are orthogonal to $n_0$.
	\end{thm}

	The far-field expansion \eqref{eq:expansion} consists of a harmonic part $n_{harm}$ solving the linearized equation $\Delta n_{harm}=0$, and of a non-harmonic correction $n_{corr}$.
	Interestingly, if the coefficient $v_0$ of the leading-order term in $n_{harm}$ vanishes, then the non-harmonic correction vanishes and 
	$n$ admits a harmonic expansion up to $\mathcal O(1/r^4)$. 
	Higher-order non-harmonic corrections would not have that property.
	This is why we stop the expansion at this order, even though it will be clear from the proof that one can obtain an expansion at any arbitrary order.
	The relations  $v_0\cdot n_0 =p_j\cdot n_0=0$ simply come from the constraint $|n|^2=1$, which also imposes similar relations about the higher order coefficients $c_{k\ell}$,
	but we do not write them  explicitly because they   do not have such a precise geometric interpretation.

	\begin{rem}\label{r:general}
		The proof of Theorem~\ref{t:expansion} can be   generalized to obtain far-field expansions for any manifold-valued map $u\colon \R^d\setminus \overline B_1\to\mathcal N\subset \R^k$ ($d\geq 3$)  with given far-field value $u_0\in\mathcal N$ in the sense $\int r^{-2}|u-u_0|^2  <\infty$, minimizing
		the Dirichlet energy.
		In the context of nematic liquid crystals with unequal elastic constants, it is interesting to consider more general energies of the form
		$\int A(u)[\nabla u,\nabla u]$, 
		where $A(u)$ is a positive definite bilinear form on $\R^{k\times n}$ depending smoothly on $u$.
		Far field asymptotics should then be dictated by the linearized system $\nabla\cdot A(u_0)\nabla v=0$, for which multipole expansions in terms of derivatives of the fundamental solution are described e.g. in \cite{BGO}.
		We expect that the tools developed in the present work will apply to that generalized setting, but do not provide the technical details here.
	\end{rem}
	
	We will obtain below
	various sufficient conditions ensuring that $v_0=0$, and so $n_{corr}=0$.
	For now, it is worth noting that $v_0$ vanishes for axisymmetric configurations. The map $n:\RR^3 \setminus \overline{B}_1 \rightarrow \mathbb{S}^2$ is axisymmetric about $n_0$ if for any rotation $R$ of axis $n_0$ one has
	\begin{align*}
		n(Rx)=Rn(x)\qquad\forall x\in\Omega.
	\end{align*}
	Using the far-field expansion \eqref{eq:expansion} in this identity implies $Rv_0=v_0$ for all rotations $R$ of axis $n_0$, and therefore $v_0=0$ since $v_0\cdot n_0=0$.
	
	\begin{cor}\label{c:symexp}
		If the minimizing map $n$ is axisymmetric about $n_0$, then $n=n_0+n_{harm}+\mathcal O(1/r^4)$ as $r=|x|\to\infty$, with $\Delta n_{harm}=0$.
	\end{cor}

	Corollary~\ref{c:symexp} is stated here for minimizing maps  that are axisymmetric, but it is hard in general to prove that a minimizing map is symmetric. However, the proof of Theorem~\ref{t:expansion} can be reproduced for an axisymmetric map which is   minimizing merely   among axisymmetric configurations (see Remark~\ref{r:expansionsym}), and Corollary~\ref{c:symexp} is valid also in that case.

	\subsection{Characterization of the leading-order term}\label{ss:introtorque}

	Next we take into account the presence of the particle, a smooth bounded open subset $G\subset\R^3$, and consider the energy
	\begin{align*}
		E(n)=\int_{\R^3\setminus G}|\nabla n|^2 + F_s(n_{\lfloor\partial G}),
	\end{align*}
	where 
	\begin{align}\label{eq:Fs}
		F_s \colon H^{1/2}(\partial G;\mathbb S^2)\to [0,\infty]
		\text{ is weakly lower semicontinuous and }\lbrace F_s   <\infty\rbrace\neq \emptyset.
	\end{align}
	This ensures that, for any $n_0\in\mathbb S^2$, 
	
	the energy $E$ admits a minimizer among maps $n\colon\R^3\setminus G \to\mathbb S^2$ such that 
	\begin{align*}
		\int_{\R^3\setminus G}\frac{|n-n_0|^2}{1+r^2} +\int_{\R^3\setminus G}|\nabla n|^2<\infty.
	\end{align*}
	To check this,
	note first that a boundary map $n_b\in  H^{1/2}(\partial G;\mathbb S^2)$ with finite surface energy $F_s(n_b)<\infty$ can be extended to a map $n\in H^1_{loc}(\R^3\setminus G;\mathbb S^2)$ such that $n\equiv n_0$ outside of a compact set using e.g. \cite[Lemma~A.1]{hkl88}, so the infimum is finite. Moreover the energy is coercive thanks to Hardy's inequality, and weakly lower semicontinuous as a sum of two weakly lower semicontinuous functions.
	Therefore we may define
	\begin{align}\label{eq:hatEn0}
		\hat E(n_0)=\min\Big\lbrace E(n)\colon &n\in H^1_{loc}(\R^3\setminus G;\mathbb S^2),
		\nonumber\\
		& \int_{\R^3\setminus G}\frac{|n-n_0|^2}{1+r^2} +\int_{\R^3\setminus G}|\nabla n|^2<\infty\Big\rbrace.
	\end{align}
	

	Examples of admissible surface energies $F_s$ include
	\begin{align*}
		F_s(n)=\begin{cases}
			0 & \text{ if }n=n_D,\\
			+\infty & \text{ otherwise,}
		\end{cases}
	\end{align*}
	for some fixed map $n_D\in H^{1/2}(\partial G;\mathbb S^2)$, which corresponds to imposing Dirichlet boundary conditions $n =n_D$ on $\partial G$; or
	\begin{align*}
		F_s(n)=\int_{\partial G} g(n,x)\, d\mathcal H^2(x),
	\end{align*}
	for some measurable function $g\colon \mathbb S^2 \times \partial G\to [0,\infty)$ which is continuous with respect to $n$; for instance $g(n,x)=|n-n_D(x)|^2$ which relaxes Dirichlet boundary conditions (strong anchoring) to  weak anchoring.

	Our second main result relates the vector $v_0$ appearing in the leading-order term of the expansion \eqref{eq:expansion}  to the gradient of the function $\hat E$ at $n_0$.

	\begin{thm}\label{t:torque}
		Let $F_s\colon H^{1/2}(\partial G;\mathbb S^2)\to [0,\infty]$ satisfy \eqref{eq:Fs}.
		Then the function $\hat E$ defined by \eqref{eq:hatEn0} is Lipschitz, and for a.e. $n_0\in\mathbb S^2$ we have
		\begin{align}\label{eq:torque}
			\nabla \hat E(n_0)=-8\pi v_0,
		\end{align}
		where $v_0=\lim_{r\to\infty} r(n-n_0)$ for any minimizing $n$ such that $\hat E(n_0)=E(n)$. 
		Moreover $\hat E$ is 
		{
			semiconcave: for all
		}
		$n_0,m_0\in\mathbb S^2$ and $v_0=\lim_{r\to\infty} r(n-n_0)$ for any minimizer $n$ achieving $\hat E(n_0)$, we have the one-sided inequality
		\begin{align*}
			\hat E(m_0)\leq \hat E(n_0)-8\pi v_0\cdot (m_0-n_0) + C |m_0-n_0|^2,
		\end{align*}
		for some constant $C=C(G,F_s)\geq 0$.
	\end{thm}
	
	\begin{rem}\label{r:BdG}
		Formula \eqref{eq:torque} relates $v_0$ to the torque applied by the particle $G$ on the nematic, in agreement with
		formal calculations in \cite{brocharddegennes70} for an axisymmetric particle.
		These formal calculations 
		can be made rigorous (and then they show that $\hat E$ is differentiable everywhere)
		if one knows that the minimization problem \eqref{eq:hatEn0} admits a unique minimizer $n$ which moreover depends smoothly on $n_0$.  
		Such uniqueness and smoothness results seem very hard to obtain in general, 
		and we use a somewhat different method to prove \eqref{eq:torque} and Theorem~\ref{t:torque}.
	\end{rem}

	Different minimizers $n$ in \eqref{eq:hatEn0}  may a priori have different asymptotic expansions \eqref{eq:expansion}. However, a crucial nontrivial consequence of Theorem~\ref{t:torque} is that at any differentiability point $n_0$ of $\hat E$, the coefficient $v_0$ of the leading-order term is uniquely determined by $n_0$, even though \eqref{eq:hatEn0} may have several minimizers. 
	We do not know whether $\hat E$ can have non-differentiable points, and whether $v_0$ can be multivalued at such points.
	The semiconcavity inequality in Theorem~\ref{t:torque} implies that all possible values of $v_0$ are included in the subdifferential of $-\frac{1}{8\pi}\hat E$. 
	It would be interesting to characterize values of $v_0$ in terms of this subdifferential.

	One may pose an analogous question for $\mathbb{S}^1$-valued minimizers in exterior domains $\R^2\setminus G$ in the plane which approach a constant $n_0=e^{i\phi_0}$ at infinity.  However the situation is completely different, 
	because finite-energy configurations don't exist in general. One way around that issue is to relax the $\mathbb S^1$-valued constraint via a Ginzburg-Landau approximation. This approach is implemented in \cite{ABGS}, with the asymptotic value $n_0=e^{i\phi_0}$ left free.
	
	An interesting consequence of the semiconcavity of $\hat E$ is that it must be differentiable, of zero gradient, at any local minimum point.
	
	\begin{cor}\label{c:eqn0}
		If $n_0\in\mathbb S^2$ is locally minimizing for $\hat E$, then $v_0=0$ and $n=n_{harm}+\mathcal O(1/r^4)$ as $r=|x|\to\infty$ with $\Delta n_{harm}=0$,
		for any minimizing $n$ such that $E(n)=\hat E(n_0)$.
	\end{cor}
	
	\begin{rem}\label{r:eqn0}
		In the physical system it is formally equivalent to rotate the far-field alignment $n_0$ or the  particle $G$. Hence Corollary~\ref{c:eqn0} tells us that, 
		when the particle is in a stable equilibrium position, 
		all minimizing configurations $n$ have a far-field expansion which is harmonic up to $\mathcal O(1/r^4)$, and whose leading order is given by the harmonic term $\sum_j p_j\partial_j(1/r)$ for some vectors $p_j\in n_0^\perp$. Such leading-order term corresponds to solutions of the equation
		\begin{align*}
			\Delta n =\frac{1}{4\pi}\sum_{j=1}^3 p_j\partial_j \delta\qquad\text{in }\R^3,
		\end{align*}
		where the singular source term can be interpreted as a dipole-moment, as described e.g. in \cite{lubensky98}.
	\end{rem}
	
	Another remarkable consequence of Theorem~\ref{t:torque} 
	concerns the important case where the particle $G$, 
	together with its anchoring properties described by the surface energy $F_s$,
	possess some rotational symmetry. 
	As mentioned earlier, we may not necessarily infer the same symmetry for all minimizers, 
	but we can make some strong geometrical conclusions concerning  the vector $v_0$ in the expansion \eqref{eq:expansion} of minimizers.
	To make this precise, we define the symmetry group of the particle (and its anchoring properties) $(G,F_s)$ as a subgroup of the orthogonal transformations $O(3)$ given by
	\begin{align*}
		\mathrm{Sym}(G, F_s)=\Big\lbrace R\in O(3)\colon & RG=G,\text{ and } \\
		& F_s(Rn\circ R^{-1})=F_s(n)\;\forall n\in H^{1/2}(\partial G;\mathbb S^2)\Big\rbrace.
	\end{align*}
	For any symmetry-preserving transformation $R\in\mathrm{Sym}(G, F_s)$, the energy $E$ is conserved under the transformation $n\mapsto Rn\circ R^{-1}$, and therefore $\hat E(n_0)=\hat E(Rn_0)$.
	
	\begin{cor}\label{c:sym}
		If the particle 
		has an axis of symmetry $\mathbf u\in\mathbb S^2$, i.e. $\mathrm{Sym}(G, F_s)$ contains all rotations $R\in SO(3)^{\mathbf u}$ about axis $\mathbf u$, then for almost all $n_0\in\mathbb S^2$ we have
		\begin{align}\label{eq:cor1.9}
			v_0(n_0)\cdot (\mathbf u\times n_0)=0,
		\end{align} 
		where $v_0(n_0)=\lim_{r\to\infty}r(n-n_0)$ for any minimizing map $n$ achieving $\hat E(n_0)$. If $\hat E$ is differentiable at $\mathbf u$ then $v_0(\mathbf u)=0$.
		
		If the particle is spherically symmetric, i.e. $\mathrm{Sym}(G, F_s)$ contains  all rotations $SO(3)$, then $v_0(n_0)=0$ for all $n_0\in\mathbb S^2$.
	\end{cor}

	Note that since $v_0$ is orthogonal to $n_0$, if $\mathbf u$ and $n_0$ are not parallel, then the identity $v_0 \cdot (\mathbf u\times n_0)=0$ forces $v_0$ to belong to a fixed line determined by $n_0$ and $\mathbf u$.
	This link between symmetry properties of $G$ and of $v_0$ gives a rigorous justification to assertions in \cite[\S~II.1.a]{brocharddegennes70} where this is deduced from the assumption, false in general, that minimizers $n$ in \eqref{eq:hatEn0} are unique.

	\begin{rem}\label{r:axisym}
		In the axisymmetric setting,
		Corollary~\ref{c:sym} leaves open the case when $\hat E$ is not differentiable at $n_0=\mathbf u$, the axis of symmetry:  the $1/r$ asymptotic  might be nonzero.
		If that situation occurs, that is,
		there is a minimizer $n$ with far-field alignment $\mathbf u$ but with $v_0\neq 0$, then all its axial rotations $R n\circ R^{-1}$ are minimizers for $\hat E(\mathbf u)$ too, with $1/r$ asymptotic term equal to $Rv_0$. The semiconcavity inequality
		\begin{align*}
			\hat E(n_0)\leq \hat E(\mathbf u)-8\pi Rv_0 \cdot (n_0-\mathbf u) + C |n_0-\mathbf u|^2,
		\end{align*}
		is then valid for all rotations $R$ of axis $\mathbf u$, and we deduce
		\begin{align*}
			\hat E(n_0)\leq \hat E(\mathbf u) -8\pi|n_0-\mathbf u| +C |n_0-\mathbf u|^2.
		\end{align*}
		Hence $\hat E$ has a local maximum at $\mathbf u$, and its graph near $\mathbf u$ looks locally like a cone.  While none of the results above preclude this scenario in the axisymmetric setting, it is natural to ask the open question: 
		can this situation really occur? 
	\end{rem}

	\subsection{Plan of the article}
	In section~\ref{s:expansion} we prove Theorem~\ref{t:expansion} and in section~\ref{s:torque} we prove Theorem~\ref{t:torque}.   In  Appendix~\ref{a:decay} we provide proofs of some familiar (but not easily found) decay estimates for Poisson's equation for the reader's convenience.

	\section{Far-field expansion}\label{s:expansion}
	In this section we prove Theorem~\ref{t:expansion}. The minimizing map $n$ solves, in the weak sense, the harmonic map equation
	\begin{align}\label{eq:eulerlagrange}
		-\Delta n=|\nabla n|^2n\qquad\text{in }\R^3\setminus \overline B_1.
	\end{align}
	If the right-hand side decays like $\mathcal O(1/|x|^\gamma)$ for some $\gamma>3$,   decay estimates for the Poisson equation (see Lemma~\ref{l:decaycorrecptwise}) enable one to start a harmonic expansion for $n$, and this process can then be iterated including relevant non-harmonic corrections. 
	Hence the main new ingredient in the proof of Theorem~\ref{t:expansion} is to obtain an initial strong enough decay estimate on $|\nabla n|$.

	Note that, since $\int_{|x|\geq R}|\nabla n|^2\to 0$ as $R\to\infty$, 
	small energy estimates for harmonic maps \cite{schoen84,schoenuhlenbeck82} ensure that $n$ is smooth outside of a finite ball of large enough radius. Specifically, given $x_0\in\R^3$, $\vert x_0\vert =R$, the small energy regularity estimate for harmonic maps \cite[Theorem~2.2]{schoen84} applied to $\hat n(\hat x)=n (x_0+(R/2)\hat x)$
	implies the existence of $R_0\geq 1$ (depending on $n$)  such that
	\begin{equation}\label{eq:smallregrescaled}
		\abs{x_0}=R\geq R_0\quad\Longrightarrow \quad \abs{\nabla n}^2(x_0)\lesssim R^{-3}\int_{\frac R2\leq\abs{x}\leq \frac {3R}{2}}\abs{\nabla n}^2.
	\end{equation}
	In particular we have the decay estimate $|\nabla n(x)|^2=  o(1/|x|^3)$.  
	At this point we would like to use decay estimates of Poisson's equation from Lemma~\ref{l:decaycorrecptwise} in an iterative process to generate the far-field expansion, but the decay given in \eqref{eq:smallregrescaled} is just not enough to start applying the Lemma.  Consequently, we require an algebraic decay
	$\mathcal O(1/R^\delta)$, for some $\delta>0$,
	of the integral $\int_{|x|\geq R}|\nabla n|^2$.  
	This we obtain in  
	Lemma~\ref{l:energyimprovement} and Step~1 of Theorem~\ref{t:expansion}'s proof, 
	using the minimizing property of $n$ in order to compare the decay of that integral with the decay of the same integral for minimizers of the Dirichlet energy with values into the plane $T_{n_0}\mathbb S^2$, that is, solutions of the linearized equation $\Delta n=0$.
	
	First recall that for harmonic functions we have the following decay estimates.
	\begin{lem}\label{l:harmonicdecay}
		Let $u\colon \R^3\setminus \overline B_1\to\R$ satisfy $\int_{\R^3\setminus \overline B_1}\abs{\nabla u}^2<\infty$ and $\Delta u=0$ in $\R^3\setminus \overline B_1$. Then for all $R\geq 1$, $\hat u(\hat x)=u(R \hat x)$ satisfies
		\begin{equation*}
			\int_{\abs{\hat x}\geq 1}\abs{\nabla\hat u}^2 = \frac 1R \int_{\abs{x}\geq R}\abs{\nabla u}^2 \leq \frac{1}{R^2}\int_{\abs{x}\geq 1}\abs{\nabla u}^2.
		\end{equation*}
	\end{lem}
	\begin{proof}
		Since $u$ is harmonic and $\int_{\R^3\setminus\overline B_1}|\nabla u|^2<\infty$, its spherical harmonics expansion is of the form
		\begin{equation*}
			u(x) = u(r\omega)=u_0 + \sum_k \frac{a_k}{r^{\gamma_k}}\phi_k(\omega),
		\end{equation*}
		where we decompose $x \neq 0$ in polar coordinates as $x = r\omega, r = |x|, $ and $\omega = \frac{x}{|x|} \in \mathbb{S}^2,$ and $\lbrace \phi_k\rbrace_k$ is an $L^2(\mathbb S^2)$-orthonormal system of spherical harmonics and $\gamma_k > 0$. Then we compute
		\begin{align*}
			\int_{\abs{x}\geq R}\abs{\nabla u}^2 &= \int_{\abs{x}\geq R}\nabla\cdot(u\nabla u) =-\int_{\abs{x}=R} u \partial_r u \\
			& = \sum_k \frac{\gamma_ka_k^2}{R^{2\gamma_k+1}} \leq \frac 1R \sum_k\gamma_ka_k^2 = \frac 1R \int_{\abs{x}\geq 1}\abs{\nabla u}^2.
		\end{align*}
	\end{proof}
	
	We obtain almost the same decay for our minimizing map $n$, via the following decay improvement result. 
	The estimate obtained in Lemma~\ref{l:energyimprovement} will be needed in Step 1 of the proof of Theorem~\ref{t:expansion}.  
	After the proof of the theorem we present a second proof of that step, replacing the estimate of Lemma~\ref{l:energyimprovement} by a different approach inspired by asymptotic expansions of minimal surfaces in \cite{schoen83}. 
	Note that, as pointed out by the anonymous referee,
	this second proof
	makes use of minimality of $n$ only for the small energy estimate, and therefore it applies also to nonminimizing stationary harmonic maps \cite{bethuel93}.
	We find it worth including both proofs here,
	as their ranges of applicability to the anisotropic energies mentioned in Remark~\ref{r:general} may differ.

	\begin{lem}\label{l:energyimprovement}
		For any $\alpha<2$, there exist $\delta>0$ and $R_1>1$ such that for any $n_0\in\mathbb S^2$ and any map $n\colon \R^3\setminus \overline B_1\to \mathbb S^2$ with
		\begin{align*}
			\int_{\R^3\setminus \overline B_1} \frac{\abs{n-n_0}^2}{r^2} \lesssim \int_{\R^3\setminus \overline B_1} \abs{\nabla n}^2<\infty,
		\end{align*}
		
		which is energy minimizing, i.e. 
		\begin{align*}
			\int_{\R^3\setminus \overline B_1}\abs{\nabla n}^2\leq\int_{\R^3\setminus \overline B_1}\abs{\nabla\tilde n}^2,
		\end{align*} 
		for all $\mathbb S^2$-valued maps $\tilde n$ that agree with $n$ outside of a compact subset of $ \R^3\setminus \overline B_1$, we have
		\begin{equation*}
			\int_{\abs{x}\geq 1}\abs{\nabla n}^2\leq\delta^2\quad\Rightarrow\quad\frac 1{R_1} \int_{\abs{x}\geq R_1}\abs{\nabla n}^2\leq \frac{1}{R_1^\alpha}\int_{\abs{x}\geq 1}\abs{\nabla n}^2.
		\end{equation*}
	\end{lem}
	
	\begin{proof}[Proof of Lemma~\ref{l:energyimprovement}]
		The proof follows quite closely the strategy in \cite[Proposition~1]{luckhaus88} (see also \cite[Theorem~2.4]{hkl86}). 
		By rotational symmetry we may assume $n_0=(0,0,1)$.
		We fix $\alpha<2$. Since $T_{n_0}\mathbb S^2=n_0^\perp=\mathbb R^2\times\lbrace 0\rbrace$, thanks to Lemma~\ref{l:harmonicdecay} we may choose any $R_\star>1$ such that for any $T_{n_0}\mathbb S^2$-valued energy minimizing map $v$ in $\R^3\setminus \overline B_1$ with $\int\abs{v}^2\abs{x}^{-2}\lesssim\int\abs{\nabla v}^2<\infty$,
		\begin{equation}\label{eq:harmonicimprovement}
			\frac 1{R_\star} \int_{\abs{x}\geq R_\star}\abs{\nabla v}^2\leq \frac 14 \frac{1}{R_\star^\alpha}\int_{\abs{x}\geq 1}\abs{\nabla v}^2.
		\end{equation}
		Then we  fix $R_1=2R_\star$ and argue by contradiction, assuming Lemma~\ref{l:energyimprovement} to be false for this value of $R_1$. Hence there exist $\delta_j\to 0$  and minimizing $\mathbb S^2$-valued maps $n_j$ such that
		\begin{align*}
			&\int_{\abs{x}\geq 1}\frac{\abs{n_j-n_0}^2}{\abs{x}^2}\lesssim\int_{\abs{x}\geq 1}\abs{\nabla n_j}^2 =\delta_j^2\\
			\text{and}\quad & \frac 1{R_1} \int_{\abs{x}\geq R_1}\abs{\nabla n_j}^2 > \frac{1}{R_1^\alpha}\int_{\abs{x}\geq 1}\abs{\nabla n_j}^2.
		\end{align*}
		We set
		\begin{equation*}
			v_j:=\frac{n_j-n_0}{\delta_j},
		\end{equation*}
		so that
		\begin{equation}\label{eq:propvj}
			\int_{\abs{x}\geq 1}\frac{\abs{v_j}^2}{\abs{x}^2}\lesssim\int_{\abs{x}\geq 1}\abs{\nabla v_j}^2 =1\quad\text{and}\quad \frac 1{R_1} \int_{\abs{x}\geq R_1}\abs{\nabla v_j}^2 > \frac{1}{R_1^\alpha}\int_{\abs{x}\geq 1}\abs{\nabla v_j}^2.
		\end{equation}
		Up to a subsequence (that we do not relabel), there exists $v_\star\in H^1_{loc}(\R^3\setminus \overline B_1;\mathbb R^3)$ such that $v_j\rightharpoonup v_\star$ weakly in $H^1_{loc}$, strongly in $L^2_{loc}$, and almost everywhere. 
		Note that $v_\star(x)\in T_{n_0}\mathbb S^2$ for a.e. $x\in \R^3\setminus \overline B_1$.
		Indeed, considering a subsequence of $v_j$ converging a.e., we see that
		$v_\star(x)$ is the limit of vectors of the form $(z_j-n_0)/\delta_j$ for some $z_j\in\mathbb S^2$ and $\delta_j\to 0$, which implies first that $z_j\to n_0$, and then that $v_\star(x)\in T_{n_0}\mathbb S^2$.
		Furthermore,
		
		by lower semi-continuity,
		\begin{equation*}
			\int_{\abs{x}\geq 1}\frac{\abs{v_\star}^2}{\abs{x}^2}\lesssim\int_{\abs{x}\geq 1}\abs{\nabla v_\star}^2 \leq 1.
		\end{equation*}
		By Fubini's theorem we may moreover pick $r\in [1,2]$ such that 
		\begin{equation*}
			\int_{\abs{x}=r}\abs{\nabla v_\star}^2\lesssim 1 \, \qquad\text{and}\quad    \int_{\abs{x}=r}\abs{\nabla v_j}^2\lesssim 1.
		\end{equation*}
		By continuity of the trace operator and compactness of the embedding $H^{\frac 12}(\partial B_r)\subset L^2(\partial B_r)$ we have $\int_{\abs{x}=r}\abs{v_j-v_\star}^2\to 0$. We claim that $v_\star$ is a $T_{n_0}\mathbb S^2$-valued minimizing map in $\Omega_r=\lbrace \abs{x}>r\rbrace$. Let $v\in H_{loc}^1(\Omega_r; T_{n_0}\mathbb S^2)$ agree with $v_\star$ outside of a compact subset of $\Omega_r$. We will show that $\int\abs{\nabla v_\star}^2\leq\int\abs{\nabla v}^2$, thus proving the claim. Let 
		\begin{align*}
			\tilde v_j & =\frac{\delta_j^{-\frac 12} v}{\max(\delta_j^{-\frac 12},\abs{v})},\qquad
			\tilde n_j  = \pi_{\mathbb S^2}(n_0 + \delta_j \tilde v_j),
		\end{align*}
		where $\pi_{\mathbb S^2}$ is the orthogonal projection onto $\mathbb S^2$ 
		(well-defined in a neighborhood of it),
		so that 
		\begin{align*}
			&\abs{\nabla \tilde n_j}^2\leq \delta_j^2\left(1+O(\delta_j^{\frac 12})\right)\abs{\nabla v}^2,\qquad \tilde v_j\to v\text{ in }H^1_{loc}(\overline \Omega_r;T_{n_0}\mathbb S^2).
		\end{align*}
		Since $v=v_\star$ on $\partial B_r$ and $\int_{\abs{x}=r}\abs{v_j-v_\star}^2\to 0$, we also have
		\begin{align*}
			\gamma_j^2:=\int_{\partial B_r}\abs{v_j-\tilde v_j}^2\to 0.
		\end{align*}
		Moreover, using that $\pi_{\mathbb S^2}$ is smooth in a small neighborhood of $n_0$ and $\tilde v_j\cdot n_0=0$, we obtain
		\begin{align*}
			\tilde n_j-n_j&
			=\pi_{\mathbb S^2}(n_0+\delta_j\tilde v_j)-n_0 -\delta_j v_j 
			=\delta_j (\tilde v_j-v_j) +\mathcal O(\delta_j^2 |v|^2),
		\end{align*}
		so
		\begin{align*}
			\int_{\partial B_r}\abs{n_j-\tilde n_j}^2 \leq \delta_j^2  (\gamma_j^2 +c^2\delta_j^2),
		\end{align*}
		where $c>0$ is a constant depending on $v$.
		Luckhaus' extension lemma~\cite[Lemma~1]{luckhaus88} ensures, for any $\lambda\in (0,1)$, the existence of $\varphi_j\colon B_{(1+\lambda)r}\setminus B_r\to\mathbb R^3$ such that
		\begin{align*}
			\varphi_j=&n_j\text{ on }\partial B_r,\quad\varphi_j=\tilde n_j((1+\lambda)\cdot)\text{ on }\partial B_{(1+\lambda)r},\\
			\int_{B_{(1+\lambda)r}\setminus B_r}\abs{\nabla\varphi_j}^2&\lesssim \lambda \int_{\partial B_r}\left(\abs{\nabla n_j}^2 + \abs{\nabla\tilde n_j}^2\right) + \lambda^{-1} \int_{\partial B_r}\abs{n_j-\tilde n_j}^2\\
			&\lesssim \delta_j^2 \left( \lambda + \lambda^{-1}(\gamma_j^2+c^2\delta_j^2)\right),\\
			\sup_{B_{(1+\lambda)r}\setminus B_r}\dist^2(\varphi_j,\mathbb S^2)&\lesssim \lambda^{-1}\left(\int_{\partial B_r}\left(\abs{\nabla n_j}^2 + \abs{\nabla\tilde n_j}^2\right)\right)^{\frac 12}\left( \int_{\partial B_r}\abs{n_j-\tilde n_j}^2\right)^{\frac 12}\\
			&\quad +\lambda^{-2}\int_{\partial B_r}\abs{n_j-\tilde n_j}^2\\
			&\lesssim \delta_j^2 \left(\lambda^{-1}\gamma_j + \lambda^{-2}\gamma_j^2\right)
		\end{align*}
		Choosing $\lambda=\lambda_j=\gamma_j+c\delta_j\to 0$, we may thus define $\psi_j =\pi_{\mathbb S^2}(\varphi_j)\colon B_{(1+\lambda_j)r}\setminus B_r\to\mathbb S^2$ satisfying
		\begin{align*}
			&\psi_j=n_j\text{ on }\partial B_r,\quad\psi_j=\tilde n_j((1+\lambda_j)\cdot)\text{ on }\partial B_{(1+\lambda_j)r},\\
			\text{and }&\delta_j^{-2}\int_{B_{(1+\lambda_j)r}\setminus B_r}\abs{\nabla\psi_j}^2\to 0.
		\end{align*}
		Then we set 
		\begin{equation*}
			\hat n_j(x) =\begin{cases}
				\psi_j(x) &\quad\text{for }r\leq \abs{x} \leq(1+\lambda_j)r,\\
				\tilde n_j((1+\lambda_j)x) &\quad\text{for }\abs{x}\geq (1+\lambda_j)r.
			\end{cases}
		\end{equation*}
		Note that $\hat n_j$ agrees with $n_j$ on $\partial B_r$ and satisfies
		\begin{equation*}
			\int_{\abs{x}\geq 2}\frac{\abs{\hat n_j -n_0}^2}{\abs{x}^2} \lesssim \int_{\abs{x}\geq r}\frac{\abs{\tilde n_j -n_0}^2}{\abs{x}^2}\lesssim \delta_j^2 \int_{\abs{x}\geq r}\frac{\abs{\tilde v_j}^2}{\abs{x}^2}\lesssim \delta_j^2 \int_{\abs{x}\geq r}\frac{\abs{v}^2}{\abs{x}^2}<\infty,
		\end{equation*}
		since $v=v_\star$ outside of a compact set and $\int_{\abs{x}\geq 1}\abs{v_\star}^2\abs{x}^{-2}<\infty$. Therefore the $\hat n_j$ must have greater energy than $n_j$, hence
		\begin{align*}
			\int_{\abs{x}\geq r}\abs{\nabla v_j}^2 & =\delta_j^{-2}\int_{\abs{x}\geq r}\abs{\nabla n_j}^2  \leq \delta_j^{-2}\int_{\abs{x}\geq r}\abs{\nabla \hat n_j}^2 \\
			&\leq (1+o(1))\delta_j^{-2}\int_{\abs{x}\geq r}\abs{\nabla\tilde n_j}^2 + o(1)\\
			&\leq (1+o(1))\int_{\abs{x}\geq r}\abs{\nabla v}^2 +o(1)
		\end{align*}
		By weak lower semi-continuity of the Dirichlet energy with respect to $H^1_{loc}$ convergence we infer
		\begin{equation*}
			\int_{\abs{x}\geq r}\abs{\nabla v_\star}^2\leq \liminf \int_{\abs{x}\geq r}\abs{\nabla v_j}^2\leq \int_{\abs{x}\geq r}\abs{\nabla v}^2,
		\end{equation*}
		so that $v_\star$ is a $T_{n_0}\mathbb S^2$-valued energy minimizing map in $\Omega_r$, and moreover applying the above to $v=v_\star$ we deduce that
		\begin{equation*}
			\int_{\abs{x}\geq r}\abs{\nabla v_j-\nabla v_\star}^2\to 0.
		\end{equation*}
		In particular, since $\int_{\abs{x}\geq 1}\abs{\nabla v_j}^2=1$, \eqref{eq:propvj} implies that $\int_{\abs{x}\geq R_1}\abs{\nabla v_\star}^2>0$.
		Moreover, recalling that $r\in [1,2]$ and taking $j\to\infty$ in \eqref{eq:propvj} we obtain
		\begin{equation*}
			\frac{1}{R_1}\int_{\abs{x}\geq R_1}\abs{\nabla v_\star}^2\geq \frac{1}{R_1^\alpha}\int_{\abs{x}\geq 2}\abs{\nabla v_\star}^2,
		\end{equation*}
		hence, for $\hat v_\star(\hat x)=v_\star(2\hat x)$, recalling that $R_1=2 R_\star$ and $\alpha < 2,$ we have
		\begin{equation*}
			\frac{1}{R_\star}\int_{\abs{x}\geq R_\star}\abs{\nabla\hat v_\star}^2\geq  \frac{2^{1-\alpha}}{R_\star^\alpha}\int_{\abs{x}\geq 1}\abs{\nabla \hat v_\star}^2\geq \frac{1}{2}\frac{1}{R_\star^\alpha}\int_{\abs{x}\geq 1}\abs{\nabla \hat v_\star}^2.
		\end{equation*}
		Since $\hat v_\star$ is a $T_{n_0}\mathbb S^2$-valued energy minimizing map in $\R^3\setminus \overline B_1$ and $\int_{\abs{x}\geq 1}\abs{\nabla \hat v_\star}^2>0$, this contradicts \eqref{eq:harmonicimprovement}.
	\end{proof}

	We will plug the
	initial decay provided by Lemma~\ref{l:energyimprovement} 
	into the equilibrium equation \eqref{eq:eulerlagrange}
	in order to deduce the expansion \eqref{eq:expansion} 
	(implying  in particular \textit{a posteriori} that Lemma~\ref{l:energyimprovement} is also valid for $\alpha=2$).
	The main tool to obtain the expansion will be decay estimates for Poisson equation.   
	These estimates are familiar but not easily found in the form we require here, and so we have provided a proof in the  Appendix~\ref{a:decay}.
	With these preliminary lemmas, we are now ready to present the proof of Theorem~\ref{t:expansion}.
	
	\begin{proof}[Proof of Theorem \ref{t:expansion}]
		Let $R_1$ and $\delta$ be as in Lemma \ref{l:energyimprovement}. 
		
		\noindent\textbf{Step 1.} Picking $R_0>1$ (depending on $n$) such that $\frac{1}{R_0}\int_{\abs{x}\geq R_0}\abs{\nabla n}^2\leq \delta^2$ we may apply Lemma~\ref{l:energyimprovement} iteratively to $x\mapsto n(R_1^kR_0 x)$ for $k\geq 0$ and obtain
		\begin{equation*}
			\frac{1}{R_1^{k} R_0}\int_{\abs{x}\geq R_1^{k}R_0}\abs{\nabla n}^2\leq \frac{\delta^2}{(R_1^k)^\alpha},
		\end{equation*}
		and therefore
		\begin{equation*}
			\frac{1}{R}\int_{\abs{x}\geq R}\abs{\nabla n}^2\leq \frac{C(n,\alpha)}{R^\alpha}\quad\forall R\geq R_0(n),\alpha<2.
		\end{equation*}
		Thanks to \eqref{eq:smallregrescaled} this implies
		\begin{equation*}
			\abs{\nabla n}\leq \frac{C(n,\sigma)}{r^{2-\sigma}}\qquad\text{for } r\geq R_0(n),\sigma>0.
		\end{equation*}
		
		Here we are interested in small values of $\sigma>0$, and $C(n,\sigma)>0$ denotes a generic constant depending on $n$ and $\sigma$, whose precise value may change from line to line in the rest of the proof.
		Integrating this along radial rays yields $\abs{n-n_0}\leq C(n,\sigma)/r^{1-\sigma}$. Moreover, since $-\Delta n =\abs{\nabla n}^2n$ we have (redefining $\sigma$ appropriately)
		\begin{equation*}
			\abs{\Delta n}\leq \frac{C(n,\sigma)}{r^{4-\sigma}}\qquad\text{for } r\geq R_0(n),\sigma>0.
		\end{equation*}
		\textbf{Step 2.} 
		Applying Lemma~\ref{l:decaycorrecptwise} to $f_1 =\Delta n = \Delta (n - n_0)$ we obtain the existence of $u_1\colon \R^3\setminus B_{R_0}\to\R^3$ such that $\Delta u_1 =\Delta (n - n_0)$ and
		\begin{equation} \label{e.u1rem}
			\frac{\abs{u_1}}{r} + \abs{\nabla u_1}\leq \frac{C(n,\sigma)}{r^{3-\sigma}}\qquad\text{for } r\geq R_0(n).
		\end{equation}
		The map $n-n_0-u_1$ is harmonic in $\R^3\setminus B_{R_0}$.
		Writing down its spherical harmonics expansion, 
		we modify $u_1$ to include the part of the expansion that decays faster than $1/r$. 
		Specifically, we have
		\begin{align*}
			n-n_0-u_1 =\frac{1}{r}v_0 +\tilde u_1,
		\end{align*}
		for some $v_0\in\R^3$ and a remainder $\tilde u_1$ satisfying
		$\abs{\tilde u_1}/r + \abs{\nabla \tilde u_1}=\mathcal O(1/r^3)$.
		Therefore,  replacing $u_1$ by $u_1+\tilde u_1$ (without renaming it), we obtain
		\begin{equation} \label{e.nexp0}
			n=n_0 + \frac{1}{r}v_0 +u_1,
		\end{equation}
		with $u_1$ still satisfying \eqref{e.u1rem}. The vector $v_0$ is, a posteriori, uniquely determined by the map $n$, since $v_0=\lim_{r\to\infty} r(n-n_0)$. Moreover,
		this implies
		\begin{align*} 
			1=|n|^2 =1+ \frac{2}{r}v_0\cdot n_0 +\mathcal O\left(\frac 1{r^{2-\sigma}}\right),
		\end{align*}
		so we must  have
		\begin{equation*}
			v_0\cdot n_0 =0.
		\end{equation*}
		
		\noindent\textbf{Step 3.} With an eye toward obtaining the next term in the far-field expansion, we plug in \eqref{e.nexp0} into the harmonic maps PDE \eqref{eq:eulerlagrange}, and isolate terms that are higher order than $\mathcal O(\frac{1}{r^5})$ on the right hand side. 
		Specifically, we have
		\begin{align*}
			0=\Delta n +|\nabla n|^2n 
			&=\Delta u_1 + \frac{1}{r^4}|v_0|^2n_0 +\mathcal O\left(\frac 1{r^{5-\sigma}}\right) \\
			&=\Delta\left( u_1 + \frac{1}{r^2}\frac{|v_0|^2}{2}n_0\right) +\mathcal O\left(\frac 1{r^{5-\sigma}}\right),
		\end{align*}
		that is,
		\begin{align*}
			\Delta \left(u_1+\frac{1}{r^2}\frac{\abs{v_0}^2}{2}n_0 \right)&= f_2,
		\end{align*}
		where $f_2$ has decay rate given by $\abs{f_2}\leq C(n,\sigma)/r^{5-\sigma}$ for $r\geq R_0(n).$ By Lemma \ref{l:decaycorrecptwise}, we obtain $u_2\colon \R^3\setminus B_{R_0}\to\R^3$ such that $\Delta u_2 =f_2$ and 
		\begin{equation*}
			\frac{\abs{u_2}}{r} + \abs{\nabla u_2}\leq \frac{C(n,\sigma)}{r^{4-\sigma}}\qquad\text{for } r\geq R_0(n).
		\end{equation*}
		The map $u_1+r^{-2}\abs{v_0}^2n_0/2 -u_2$ is harmonic in $\R^3\setminus B_{R_0}$, hence including the higher decay part of its spherical harmonics expansion into $u_2$ we deduce the existence of  $P_1\in \mathbb R[X]^3$, a vector of homogeneous harmonic polynomials of degree 1 (i.e. linear forms) such that
		\begin{equation*}
			u_1 =-\frac{1}{r^2}\frac{\abs{v_0}^2}{2}n_0 + \frac{1}{r^3} P_1(x) +u_2,
		\end{equation*}
		i.e.
		\begin{equation} \label{e.nexp2}
			n=\left(1-\frac{\abs{v_0}^2}{2r^2}\right) n_0 +\frac 1r v_0 +\frac{1}{r^3} P_1(x) +u_2.
		\end{equation}
		Note that the unit norm constraint on $n$ implies $n_0\cdot P_1(x)=0$ for all $x$. Indeed, taking the norm square of \eqref{e.nexp2}, we find 
		\begin{align*}
			1 = |n|^2 = 1 + \frac{2 \, n_0 \cdot P_1(x/r)}{r^2} + \mathcal O\left(\frac{1}{r^{3-\sigma}}\right),
		\end{align*}
		which implies $n_0 \cdot P_1(x) \equiv 0$. 
		Writing $P_1(x)/r^3=\sum {p_j}\partial_j(1/r)$,
		we must have $p_j\cdot n_0=0$ for $j=1,2,3$.
		
		\noindent\textbf{Step 4.} As before, we plug \eqref{e.nexp2} back again into the equation \eqref{eq:eulerlagrange} and isolate terms that are $\mathcal O(\frac{1}{r^6} )$ on the right hand side. We find
		\begin{align*}
			0 =\Delta n+|\nabla n|^2n &=\Delta u_2 + \frac{1}{r^5}|v_0|^2v_0 
			+\frac{4}{r^6}(v_0\cdot P_1(x))\, n_0
			+ \mathcal O\left(\frac{1}{r^{6-\sigma}}\right) \\
			&=\Delta \left( u_2 + \frac{1}{6r^3}|v_0|^2v_0
			+\frac 1{3r^4}(v_0\cdot P_1(x))\, n_0
			\right) + \mathcal O\left(\frac{1}{r^{6-\sigma}}\right) .
		\end{align*}
		Applying Lemma~\ref{l:decaycorrecptwise} and arguing as in Steps~2 and 3, we deduce the existence of $P_2\in \R[X]^3$ a vector of homogeneous harmonic polynomials of degree 2 (i.e. harmonic quadratic forms) such that we have the expansion
		\begin{align*} 
			n&=\left(1-\frac{\abs{v_0}^2}{2r^2} \right) n_0
			+\frac 1r v_0 
			+\frac{1}{r^3} P_1(x) 
			-\frac{\abs{v_0}^2}{6r^3}v_0
			-\frac{1}{3r^4}
			(v_0\cdot P_1)\, n_0
			+\frac{1}{r^5} P_2(x)
			+u_3,\\
			&\frac{\abs{u_3}}{r}+\abs{\nabla u_3}\leq \frac{C(n,\sigma)}{r^{5-\sigma}}\qquad\text{for }r\geq R_0.
		\end{align*}
		With one more iteration we realize that the decay $u_3=\mathcal O(1/r^{4-\sigma})$ improves to $u_3=\mathcal O(1/r^{4})$. Writing $P_1(x)/r^3=\sum {p_j}\partial_j(1/r)$ and $P_2(x)/r^5=\sum c_{k\ell}\partial_k\partial_\ell(1/r)$,  
		the proof of Theorem~\ref{t:expansion} is complete. 
	\end{proof}

	\begin{proof}[Alternative proof of Step~1]
		We present here another proof of Step~1, inspired by \cite[Proposition~3]{schoen83}.
		The map $w=\partial_k n$ solves, for $r=|x|\geq R_0$, the system
		\begin{align}\label{eq:w}
			-\Delta w =2\nabla n:\nabla w\, n + |\nabla n|^2 w,
		\end{align}
		where for matrices $A,B$ we use the notation $A : B := \mathrm{tr}(A^TB),$ for their Frobenius inner product.   
		Testing \eqref{eq:w} with $\eta^2w$ for some smooth cut-off function $\eta$ we obtain
		\begin{align*}
			\int\eta^2|\nabla w|^2 &\lesssim \int |\eta|\,|\nabla \eta| \, |w|\,|\nabla w| 
			+ \int \eta^2 |w| |\nabla n| |\nabla w| +\int \eta^2 |\nabla n|^2 |w|^2 \\
			& \leq \frac 12 \int \eta^2 |\nabla w|^2 +C \int |\nabla\eta|^2|w|^2 + \eta^2|\nabla n|^2 |w|^2.
		\end{align*}
		Absorbing the first term of the last line in the left-hand side, choosing $\mathbf 1_{R\leq |x|\leq 2R}\leq\eta \leq \mathbf 1_{R/2\leq |x|\leq 3R}$ with $|\nabla\eta|\lesssim 1/R$,
		and using $|w|^2\leq |\nabla n|^2\lesssim 1/r^{3}$ thanks to  \eqref{eq:smallregrescaled}, we deduce
		\begin{align*} 
			\int_{R\leq |x|\leq 2R}|\nabla w|^2 &\lesssim \frac{1}{R^2},
		\end{align*}
		hence
		\begin{align}\label{e.decay1}
			\int_{|x|\geq R}|\nabla w|^2 \leq \sum_{k\geq 0}\int_{2^k R\leq |x|\leq 2^{k+1}R}|\nabla w|^2 \lesssim \sum_{k\geq 0}\frac{1}{2^{2k}R^2}\lesssim \frac{1}{R^2}
		\end{align}
		
		Therefore,  plugging in \eqref{eq:smallregrescaled} and \eqref{e.decay1} in \eqref{eq:w}, 
		we find that the right-hand side of \eqref{eq:w}  has $\mathcal O(R^{-4})$ decay in an appropriate $L^2$ sense. To be precise, 
		\begin{align*}
			-\Delta w =f,\qquad \left(\frac{1}{R^3}\int_{ |x|\geq R}|x|^2|f|^2\right)^{\frac 12}\lesssim \frac{1}{R^3}.
		\end{align*}
		Applying Lemma~\ref{l:decay}  with the choice $\gamma = 3 - \sigma/2$ for any small $\sigma > 0,$ we deduce the existence of a map $u$ such that $-\Delta u=f$ and
		\begin{align*}
			\left(\frac{1}{R^3}\int_{|x|\geq R}\frac{|u|^2}{|x|^2}\right)^{\frac 12}\lesssim \frac{1}{R^{3-\sigma/2}},
		\end{align*}
		which implies
		\begin{align*}
			\int_{|x|\geq R}|u|^2 \leq \sum_{k\geq 0} 2^{2k+2}R^2\int_{2^k R\leq |x|\leq 2^{k+1}R} \frac{|u|^2}{|x|^2} \lesssim \sum_{k\geq 0}\frac{1}{2^{(1-\sigma)k}}\frac{1}{R^{1-\sigma}}\lesssim \frac{1}{R^{1-\sigma}},
		\end{align*}
		for any $\sigma>0$.
		
		Since $w-u$ is harmonic and square integrable at $\infty$, we have $w-u=\mathcal O(1/r^2)$ as $r\to\infty$, and deduce from this and the above that
		\begin{align*}
			\int_{|x|\geq R}|w|^2 \lesssim\frac{1}{R^{1-\sigma}}.
		\end{align*}
		Recalling $w=\partial_k n$ this implies, together with \eqref{eq:smallregrescaled}, $|\nabla n|^2\lesssim 1/r^{4-\sigma}$ and the iteration starting in Step~2 of Theorem~\ref{t:expansion}'s proof can now be applied.
	\end{proof}

	\begin{rem}\label{r:expansionsym}
		We sketch here how to modify the proof of Theorem~\ref{t:expansion} for maps $n$ which are minimizing only among axisymmetric configurations, so Corollary~\ref{c:symexp} applies also in that case.
		First of all, $n$ is smooth outside of a large finite ball thanks to small energy estimates which are valid also in that setting: see e.g. \cite[Lemma~4.1]{HKL90} where the symmetry condition is slightly more restrictive but the proof can be adapted, or note that $n$ is stationary harmonic thanks to the methods in \cite[\S~2.1]{DMP21} and apply \cite[Theorem~I.4]{bethuel93}.
		Then the alternative proof of Step~1 applies without modification, as do the rest of the steps. The first proof of Step~1 can also be applied, with the constraint that the constructed comparison map needs to be axisymmetric.
	\end{rem}

	\section{The leading-order term} \label{s:torque}
	
	In this section we prove Theorem~\ref{t:torque} and Corollary~\ref{c:sym}.

	\begin{proof}[Proof of Theorem~\ref{t:torque}]
		Without loss of generality assume $G\subset B_1$ and fix a $C^1$ function $\chi\colon \mathbb R^3\to [0,1]$ such that $\chi\equiv 0$ on $B_1$ and $\int_{|x|\geq 1} |x|^{-2} (\chi-1)^2\, dx\lesssim \int_{|x|\geq 1} |\nabla \chi|^2\, dx <\infty$.
		Here, as stated in the introduction, $\lesssim$ denotes inequality up to an absolute constant, the cut-off function $\chi$ being fixed.
		In what follows, for any $m_0 \in \mathbb{S}^2$,
		we denote by 
		\begin{align*}
			H(m_0) = \left\{ m\in H^1_{loc}(\R^3\setminus G;\mathbb S^2)\colon  \int_{\RR^3\setminus G} \frac{|m-m_0|^2}{1+ r^2} + \int_{\R^3\setminus G} |\nabla m|^2   + F_s(m_{\lfloor \partial G})<\infty \right\},
		\end{align*}
		the class of admissible competitors in the minimization problem \eqref{eq:hatEn0} defining $\hat E(m_0)$. 
		This class depends also on $\partial G$ and $F_s$, which remain fixed throughout the proof.
		
		\medskip
		{\bf Step 1:} The map $\hat E$ is Lipschitz.
		\medskip
		
		Let $n_1,n_2$ be minimizers with far-field alignments $n_1^\infty,n_2^\infty$. 
		For any angle $\vartheta\in\R$, we denote by $\mathcal R(\vartheta)\in SO(3)$ the rotation of axis $e_1$ and angle $\vartheta$.
		We choose the frame such that $n_1^\infty=e_3$ and $n_2^\infty =\mathcal R(\theta) e_3$, where $\theta$ is an angle satisfying $|n_1^\infty-n_2^\infty|\leq \theta \leq 2 |n_1^\infty-n_2^\infty|$.
		Consider now the map 
		$\tilde n_1\in H(n_1^\infty)$ given by
		\begin{align*}
			\tilde n_1(x) = \mathcal R(-\chi(x)\theta)\, n_2(x).
		\end{align*}
		We have
		\begin{align*}
			|\nabla\tilde n_1|^2 &\leq \theta^2 |\nabla\chi|^2 + |\nabla n_2|^2 + 2 \theta \, |\nabla \chi|\,|\nabla n_2| \\
			&\leq (1+\lambda^{-1})\theta^2 |\nabla\chi|^2 + (1+\lambda)|\nabla n_2|^2 ,
		\end{align*}
		for any $\lambda>0$, hence
		\begin{align*}
			\hat E(n_1^\infty)\leq C(1+\lambda^{-1})|n_1^\infty -n_2^\infty|^2 + (1+\lambda) \hat E(n_2^\infty)  .
		\end{align*}
		Applying this to $\lambda=1$ and a fixed $n_2^\infty$ we deduce in particular that $\hat E$ is bounded on $\mathbb S^2$. Moreover, choosing $\lambda=|n_1^\infty-n_2^\infty|$ we obtain
		\begin{align*}
			\hat E(n_1^\infty)-\hat E(n_2^\infty)\leq |n_1^\infty -n_2^\infty| \left( \hat E(n_2^\infty) + C + C |n_1^\infty -n_2^\infty|\right).
		\end{align*}
		Reversing the roles of $n_1,n_2$ and recalling that $\hat E(n^\infty)$ is bounded on $\mathbb{S}^2$, we conclude that $\hat E$ is Lipschitz.
		
		\medskip
		{\bf Step 2:} 
		At every differentiability point $n_0\in\mathbb S^2$ of $\hat E$ we have
		$\nabla \hat E(n_0)=-8\pi v_0$, where $v_0=\lim_{r\to\infty} r(n-n_0)\in T_{n_0}\mathbb S^2$ for any minimizer $n$ such that $E(n)=\hat E(n_0)$.
		Here recall that $r=|x|$ and the limit $v_0$ is well-defined for any such map $n$, thanks to Theorem~\ref{t:expansion}.
		
		\medskip
		
		Let $n_0\in\mathbb S^2$ be a differentiability point of $\hat E$. For any axis $e\in\mathbb S^2$ let $\mathcal R(\theta)$ be the rotation of axis $e$ and angle $\theta$, and set $n_\theta^\infty=\mathcal R(\theta) n_0$, so that
		\begin{align*}
			\hat E(n_\theta^\infty)-\hat E(n_0)=\nabla \hat E(n_0)\cdot (\mathcal R'(0) n_0) + o(\theta)\qquad\text{as }\theta\to 0.
		\end{align*}
		Define $\tilde n\in H( n_\theta^\infty)$ by $\tilde n=\mathcal R(\chi\theta) n$,  where $n$ is a minimizer such that $E(n)=\hat E(n_0)$. Using the equation satisfied by $n$ and the fact that $\tilde n=n$ in $\partial G$, for all $R>1$ we have
		\begin{align*}
			&
			\int_{B_R\setminus G} |\nabla \tilde n|^2-\int_{B_R\setminus G}|\nabla n|^2
			\\
			&=
			\int_{B_R\setminus G} \left(2\nabla n\cdot \nabla(\tilde n-n) +|\nabla (\tilde n-n)|^2\right)\\
			&=2\int_{\partial B_R}\partial_r n\cdot (\tilde n-n) 
			+\int_{B_R\setminus G} \left(-2\Delta n \cdot (\tilde n-n) 
			+|\nabla (\tilde n-n)|^2\right)\\
			&=2\int_{\partial B_R}\partial_r  n\cdot (\tilde n-n)
			+ \int_{B_R\setminus G} 2|\nabla n|^2 n \cdot (\tilde n-n) 
			+ \int_{B_R\setminus G} |\nabla (\tilde n-n)|^2\\
			&=2\int_{\partial B_R}\partial_r n\cdot (\tilde n-n)
			- \int_{B_R\setminus G} |\nabla n|^2 |\tilde n-n|^2 
			+ \int_{B_R\setminus G} |\nabla (\tilde n-n)|^2
		\end{align*}
		Using the asymptotic expansion of the minimizing map $n$ 
		($n=n_0 +v_0/r +u_1$, see \eqref{e.nexp0}, with $|u_1|/r + |\nabla u_1|=\mathcal O(1/r^3)$ thanks to \eqref{e.nexp2}) we have
		\begin{align*}
			\int_{B_R}\partial_r n\cdot (\tilde n-n)=-8\pi v_0\cdot (n_\theta^\infty-n_0) + O(1/R)\quad\text{as } R\to\infty,
		\end{align*}
		where $v_0=\lim_{r\to\infty} r(n-n_0)\in T_{n_0}\mathbb S^2$. We deduce that
		\begin{align} \label{e.locsemconc}
			&\hat E(n_\theta^\infty)-\hat E(n_0) \nonumber\\
			&\leq E(\tilde n)- E(n)
			=\lim_{R\to\infty} \left(
			\int_{B_R\setminus G} |\nabla \tilde n|^2-\int_{B_R\setminus G}|\nabla n|^2
			\right)
			\nonumber\\
			&=-8\pi v_0 \cdot (n_\theta^\infty-n_0)
			- \int_{\R^3\setminus G} |\nabla n|^2 |\tilde n-n|^2 
			+ \int_{\R^3\setminus G} |\nabla (\tilde n-n)|^2 \nonumber\\
			& \leq  -8\pi v_0 \cdot (n_\theta^\infty-n_0) + C\left(1+\int_{\R^3\setminus G}|\nabla n|^2 \right)  \theta^2.
		\end{align}
		The last estimate follows from the explicit form of $\tilde n=\mathcal R(\chi\theta) n$, and the constant $C$ depends only on the fixed cut-off function $\chi$. In particular we have
		\begin{align*}
			\hat E(n_\theta^\infty)-\hat E(n_0) \leq -8\pi v_0\cdot (\mathcal R'(0) n_0) + O(\theta^2),
		\end{align*}
		which implies
		\begin{align*}
			(\nabla \hat E(n_0)+8\pi v_0)\cdot (\mathcal R'(0) n_0)\leq 0.
		\end{align*}
		Since $\mathcal R'(0) n_0$ can be any tangent vector in $T_{n_0}\mathbb S^2$ we infer that $\nabla \hat E(n_0)+8\pi v_0=0$.
		\\
		\textbf{Step 3.} It remains to prove that $\hat E$ is semiconcave. This follows directly from the inequality \eqref{e.locsemconc} obtained in Step 2, as any $m_0\in\mathbb S^2$ can be written as $m_0=n_\theta^\infty$ for some $0\leq\theta\leq 2|m_0-n_0|$.
		This completes the proof of  Theorem~\ref{t:torque}.
	\end{proof}

	\begin{proof}[Proof of Corollary \ref{c:sym}]
		Consider first the axisymmetric case $\mathrm{Sym}(G)\supset SO(3)^{\mathbf u}$. Then we have $\hat E(Rn_0)=\hat E(n_0)$ for any rotation $R$ of axis $\mathbf u$ and $n_0\in\mathbb S^2$. At a differentiable point $n_0$, differentiating this identity with respect to $R$ implies $\nabla \hat E(n_0)\cdot An_0=0$ for any antisymmetric matrix $A$ with $A\mathbf u=0$, i.e. $\nabla \hat E(n_0)\cdot (\mathbf u\times n_0)=0$. 
		Recalling from Theorem~\ref{t:torque} that $\nabla \hat E(n_0)=-8\pi v_0$, we deduce $v_0\cdot (\mathbf u\times n_0)=0$.

		Moreover, if $\mathbf u$ is a differentiability point, then differentiating that same identity
		with respect to $n_0$ at $n_0=\mathbf u$ gives $R^{-1}\nabla  \hat E(\mathbf u)=\nabla\hat E(\mathbf u)$ for any rotation $R$ of axis $\mathbf u$, hence $\nabla \hat E(\mathbf u)=0$ since $\nabla \hat E(u)\in T_{\mathbf u}\mathbb S^2=\mathbf u^\perp$.  So $v_0(\mathbf u)=0$.
		
		In the spherically symmetric case $\mathrm{Sym}(G)\supset SO(3)$ we have $\hat E(Rn_0)=\hat E(n_0)$ for all $R\in SO(3)$, hence $\hat E$ is constant, and $\nabla\hat E=0$ on $\mathbb S^2$.
		So $v_0(n_0)=0$ for all $n_0\in\mathbb S^2$.
	\end{proof}
	
	\section*{Acknowledgements}
	S.A. and L.B. were supported via an NSERC (Canada) Discovery Grant.  X.L. received  support  from ANR project ANR-18-CE40-0023. The work of R.V. was partially supported by a grant from the Simons Foundation (award \# 733694) and an AMS-Simons travel award. 
	We wish to thank the anonymous referees for the many substantial improvements they suggested.

	\appendix
	
	\section{Decay estimates for Poisson's equation} \label{a:decay}
	
	We collect here some folklore decay estimates for Poisson's equation. For the reader's convenience we include  a self-contained proof (similar arguments can be found e.g. in \cite[\S~2.2.3]{pacardriviere} for H\"older decay at the origin). 
	The elementary arguments we present here don't seem to apply directly for general systems as in Remark~\ref{r:general}, in that case one should refer to \cite[\S~5-6]{BGO}.

	\begin{lem}\label{l:decaycorrecptwise}
		Let $d\geq 3$, $\gamma >d-2$, $\gamma\notin\mathbb N$,  and  $f$ a function in $\R^d\setminus B_1$ satisfying 
		\begin{align*}
			|f(x)|\leq \frac{1}{r^{\gamma +2}}\qquad\text{for }r=|x|\geq 1.
		\end{align*}
		Then there exists a function $u$ such that $\Delta u =f$ in $\R^d\setminus B_1$ and
		\begin{align}\label{eq:uptwisebound}
			\frac{|u(x)|}{r} +|\nabla u(x)|\lesssim \frac{1}{r^{\gamma +1}},
		\end{align}
		where the constant depends only on $d$ and $\gamma$.
	\end{lem}

	Note that \eqref{eq:uptwisebound} doesn't determine $u$ uniquely, as we may add any faster-decaying harmonic terms to $u$ without changing the equation $\Delta u=f$, but the proof does determine an explicit  right inverse $f\mapsto u$ to the Laplacian in that decay range.

	We will obtain Lemma~\ref{l:decaycorrecptwise} as a consequence of an $L^2$ version of it, that we state now.

	\begin{lem}\label{l:decay}
		Let $d\geq 3$, $\gamma >d-2$, $\gamma\notin\mathbb N$, and  $f$ a function in $\R^d\setminus B_1$ satisfying 
		\begin{align} \label{e.fdecay}
			\left(\frac{1}{R^d}\int_{\abs{x}> R} \abs{x}^2 f^2 \, dx  \right)^{\frac 12}&\leq \frac{1}{R^{\gamma+1}}\qquad\forall R\geq 1,
		\end{align}
		Then there exists a function $u$ such that $\Delta u =f$ in $\R^d\setminus B_1$ and
		\begin{align} \label{e.udecay}
			\left(\frac{1}{R^d}\int_{\abs{x}\geq R}\frac{\abs{u}^2}{|x|^2}\, dx\right)^{\frac 12}
			&\lesssim \frac{1}{R^{\gamma+1}}\qquad\forall R\geq 1,
		\end{align}
		where the implicit constant depends only on $d$ and $\gamma$.
	\end{lem}

	Before proving Lemma~\ref{l:decay}, we explain why, together with rescaled elliptic estimates, it implies Lemma~\ref{l:decaycorrecptwise}.
	
	\begin{proof}[Proof of Lemma~\ref{l:decaycorrecptwise}]
		The assumption on $f$ implies that it satisfies the $L^2$ decay in the assumption of Lemma~\ref{l:decay}, so we obtain $u$ such that $\Delta u=f$ in $\R^d\setminus B_1$ and
		\begin{align*}
			\left(\frac{1}{R^d}\int_{\abs{x}\geq R}\frac{\abs{u}^2}{|x|^2}\, dx\right)^{\frac 12}
			\lesssim
			\frac{1}{R^{\gamma+1}}\qquad\forall R\geq 1,
		\end{align*}
		and the pointwise bound \eqref{eq:uptwisebound} in the conclusion of Lemma~\ref{l:decaycorrecptwise} follows from rescaled elliptic estimates.
		Explicitly, consider $\hat u(\hat x)=u(R\hat x)$ which solves $\Delta \hat u = \hat f$, where $\hat f(\hat x):=R^2 f(R\hat x)$, then from interior elliptic estimates (see e.g. \cite{gilbargtrudinger}) we have
		\begin{align*}
			\sup_{B_3\setminus B_2}\left( |\hat u| + |\nabla \hat u| \right)
			& \lesssim \left(\int_{B_4\setminus B_1}|\hat u|^2  \right)^{\frac 12} + \sup_{B_4\setminus B_1}|\hat f|\\
			&\lesssim R \left(\frac{1}{R^d}\int_{|x|\geq R}\frac{|u|^2}{|x|^2}  \right)^{\frac 12}+\frac{1}{R^{\gamma}},
		\end{align*}
		from which, scaling back, we infer \eqref{eq:uptwisebound}.
	\end{proof}

	Next we prove Lemma~\ref{l:decay}. Before doing so,
	we  recall some facts concerning spherical harmonics 
	(that is, homogeneous harmonic polynomials), referring the reader to \cite{Stein} for details.
	The Laplace-Beltrami operator on $\mathbb S^{d-1}$ diagonalizes as
	\begin{equation*}
		-\Delta_{\mathbb S^{d-1}}\Phi_j =\lambda_j \Phi_j,\qquad 0= \lambda_0 \leq \lambda_1\leq\cdots
	\end{equation*}
	The set $\lbrace \lambda_j\rbrace_{j\in\mathbb N}$ coincides with $\lbrace k^2 + k(d-2)\rbrace_{k\in\mathbb N}$. 
	The eigenfunctions corresponding to $k^2 + k(d-2)$ span the homogeneous harmonic polynomials 
	of degree $k$. 
	We choose them normalized in $L^2(\mathbb S^{d-1})$ so they form an orthonormal Hilbert basis of this space. 
	For  a $W^{2,2}_{loc}$ function $w\colon (0,\infty) \to \RR$ we have
	\begin{equation} \label{e.Lj}
		\Delta (w(r)\Phi_j(\omega)) =(\mathcal L_j w )(r) \Phi_j(\omega),\qquad\mathcal L_j =\partial_{rr} +\frac{d-1}{r}\partial_r -\frac{\lambda_j}{r^2}.
	\end{equation}
	The solutions of $\mathcal L_j w=0$ are linear combinations of $r^{\gamma_j^+}$ and $r^{-\gamma_j^-}$, where $\gamma_j^\pm\geq 0$ are given by
	\begin{align*}
		\gamma_j^+ &= \sqrt{\left(\frac{d-2}{2}\right)^2 +\lambda_j} - \frac{d-2}{2}   = k \qquad &\text{for }\lambda_j =k^2+k(d-2),\\
		\gamma_j^- &= \sqrt{\left(\frac{d-2}{2}\right)^2 +\lambda_j} + \frac{d-2}{2}   = k +d-2 \qquad &\text{for }\lambda_j =k^2+k(d-2).
	\end{align*}
	The decay rate $\gamma >d-2$, $\gamma\notin\mathbb N$, is fixed and we denote by $j_0=j_0(\gamma)$ the integer $j_0\geq 0$ such that
	\begin{align*}
		&\left\lbrace j\in \mathbb N \colon \gamma_j^- <\gamma\right\rbrace
		=\lbrace 0,\ldots, j_0\rbrace,\\
		& \left\lbrace j\in \mathbb N \colon \gamma_j^- > \gamma\right\rbrace
		=\lbrace j_0 + 1, j_0 +2,\ldots \rbrace.
	\end{align*}
	
	\begin{proof}[Proof of Lemma~\ref{l:decay}]
		We extend $f$ to be defined in $\R^d$, with the property that 
		\begin{align*}
			\left(\int_{\abs{x}\leq 1} \abs{x}^2f^2 \, dx\right)^{\frac 12} &\leq 1,
		\end{align*}
		and will construct a function $u$ such that $\Delta u=f$ in $\R^d\setminus \lbrace 0\rbrace$. The function  $f\in L^2(\R^d)$ admits a spherical harmonics expansion
		\begin{equation*}
			f =\sum_{j\geq 0} f_j(r)\Phi_j(\omega),
		\end{equation*}
		and the decay assumption \eqref{e.fdecay} on $f$ amounts to
		\begin{align}\label{eq:fjdecay}
			\sum_{j\geq 0}\int_R^\infty  f_j(r)^2r^{d+1}\, dr \leq R^{d-2\gamma-2}.
		\end{align}
		We define $u$ as 
		\begin{align*}
			u:=\sum_{j\geq 0} u_j(r)\Phi_j(\omega),
		\end{align*}
		where $u_j\in W^{2,2}_{loc}(0,\infty)$ satisfy
		\begin{align*}
			\mathcal L_j u_j =f_j.
		\end{align*}
		To write down an explicit formula for $u_j$ we rewrite $\mathcal L_j$, defined in \eqref{e.Lj}, as
		\begin{align*}
			\mathcal L_j u  =r^{-d+1+\gamma_j^-}\partial_r  [  r^{d-1-2\gamma_j^-}
			\partial_r ( r^{\gamma_j^-}u) ],
		\end{align*}
		and define
		\begin{align} \label{e.ujdef}
			u_j(r)&=
			\left\lbrace
			\begin{aligned}
				r^{-\gamma_j^-}\int_r^\infty t^{2\gamma_j^-+1-d}\int_t^\infty s^{d-1-\gamma_j^-} f_j(s)\, ds\, dt &\qquad\text{if }j\in\lbrace 0,\ldots, j_0\rbrace,\\
				r^{-\gamma_j^-}\int_0^r t^{2\gamma_j^-+1-d}\int_t^\infty s^{d-1-\gamma_j^-} f_j(s)\, ds \, dt &\qquad\text{if }j \geq j_0 +1.
			\end{aligned}
			\right.
		\end{align}
		This is well defined because  for any $t >0$ using Cauchy-Schwarz,  \eqref{eq:fjdecay}  with the choice $R = t$, and the fact that $\gamma_j^-\geq d-2>0$, we can estimate the inner integral by
		\begin{align}\label{eq:estimfj1}
			\int_t^\infty s^{d-1-\gamma_j^-} \abs{f_j(s)}\, ds
			&\leq \left(\int_t^\infty s^{-2-2\gamma_j^-}s^{d-1}ds\right)^{\frac 12}
			\left(\int_t^\infty s^2f_j(s)^2s^{d-1}ds\right)^{\frac 12}  \\
			&\leqslant \frac{1}{\sqrt{2\gamma_j^- + 2-d}}t^{\tfrac{d}{2} - \gamma_j^- - 1} t^{\tfrac{d}{2} - \gamma - 1} = \frac{1}{\sqrt{ 2\gamma_j^- + 2-d}} t^{d - \gamma - \gamma_j^- -2} . \nonumber
		\end{align}
		Furthermore, as  $t\mapsto t^{2\gamma_j^- + 1 - d }t^{d-2-\gamma - \gamma_j^-} = t^{\gamma_j^--\gamma-1}$ is integrable near $\infty$ if $\gamma_j^-<\gamma$, i.e., if  $j\leq j_0$; and is integrable near 0 if $\gamma_j^->\gamma$, i.e., if  $j\geq j_0 +1$, the functions $u_j$ in \eqref{e.ujdef} are well-defined. 
		
		Let $j\leq j_0$ and set
		\begin{align*}
			\alpha:=\gamma +\gamma_{j_0}^- +1 -d,
		\end{align*}
		so that  
		$2\gamma +1 - d >\alpha > 2\gamma_j^-+1 - d$. By \eqref{eq:estimfj1} and Cauchy-Schwarz we have
		\begin{align*}
			|u_j(r)|^2 &
			\leq \frac{r^{-2\gamma_j^-}}{2+2\gamma_j^--d}\left( 
			\int_r^\infty t^{\gamma_j^--\frac{d}{2}} 
			\left(\int_t^\infty s^2f_j(s)^2s^{d-1}ds\right)^{\frac 12}\, dt
			\right)^2 \\
			& = \frac{r^{-2\gamma_j^-}}{2+2\gamma_j^--d}\left( 
			\int_r^\infty t^{\gamma_j^--\frac{d}{2}-\frac{\alpha}{2}} 
			t^{\frac{\alpha}{2}}\left(\int_t^\infty s^2f_j(s)^2s^{d-1}ds\right)^{\frac 12}\, dt
			\right)^2 \\
			& \leq 
			\frac{r^{-2\gamma_j^-}}{2+2\gamma_j^--d}
			\int_r^\infty t^{2\gamma_j^--d-\alpha}\, dt \int_r^{\infty}t^{\alpha} \left(\int_t^\infty s^2f_j(s)^2s^{d-1}ds\right)\, dt \\
			& = \frac{r^{-d+1-\alpha}}{(2+2\gamma_j^--d)(\alpha-2\gamma_j^- +d - 1)}
			\int_r^{\infty}t^{\alpha} \left(\int_t^\infty s^2f_j(s)^2s^{d-1}ds\right)\, dt\\
			&\leq 
			\frac{r^{-d+1-\alpha}}{(d-2)(\gamma-\gamma_{j_0}^-)}
			\int_r^{\infty}t^{\alpha} \left(\int_t^\infty s^2f_j(s)^2s^{d-1}ds\right)\, dt,
		\end{align*}
		where in the last line, we used that $\gamma_j^- \geqslant d-2$ so that $ 2 + 2\gamma_j^- - d \geqslant d-2,$ and that $\gamma + \gamma_{j_0}^- - 2\gamma_{j}^- \geqslant \gamma - \gamma_{j_0}^-,$ when $j \leqslant j_0.$
		Summing and using \eqref{eq:fjdecay}, we deduce
		\begin{align*}
			\sum_{j=0}^{j_0} \frac{|u_j(r)|^2}{r^2} &
			\leq \frac{r^{-d-1-\alpha}}{(d-2)(\gamma-\gamma_{j_0}^-)}
			\int_r^{\infty}t^{\alpha} \left(\sum_{j=0}^{j_0}\int_t^\infty s^2f_j(s)^2s^{d-1}ds\right)\, dt \\
			& \le  \frac{r^{-d-1-\alpha}}{(d-2)(\gamma-\gamma_{j_0}^-)} \int_r^{\infty}t^{\alpha +d-2\gamma -2}\, dt \\
			& =\frac{r^{-2\gamma-2}}{(d-2)(\gamma-\gamma_{j_0}^-)(2\gamma +1 -d -\alpha)}\\
			& \leq \frac{r^{-2\gamma-2}}{(d-2)(\gamma-\gamma_{j_0}^-)^2}.
		\end{align*}
		Similarly, for $j\geq j_0+1$ we set 
		\begin{align*}
			\beta = \gamma +\gamma_{j_0+1}^- +1 -d,
		\end{align*}
		which satisfies $2\gamma +1 - d <\beta < 2\gamma_j^-+1 - d$.
		Using \eqref{eq:estimfj1} and Cauchy-Schwarz we find
		\begin{align*}
			|u_j(r)|^2
			&
			\leq \frac{r^{-d+1-\beta}}{(d-2)(\gamma_{j_0+1}^- -\gamma)}
			\int_0^{r}t^{\alpha_j} \left(\int_t^\infty s^2f_j(s)^2s^{d-1}ds\right)\, dt,
		\end{align*}
		so that, we similarly obtain from \eqref{eq:fjdecay} that
		\begin{align*}
			\sum_{j=j_0+1}^{\infty} \frac{|u_j(r)|^2}{r^2}
			& \leq \frac{r^{-2\gamma-2}}{(d-2)(\gamma_{j_0+1}^- -\gamma)^2}.
		\end{align*}
		We conclude that
		\begin{align*}
			\sum_{j=0}^\infty
			\frac{|u_j(r)|^2}{r^2}
			& \leq \frac{1}{d-2}\left(\frac{1}{(\gamma-\gamma_{j_0}^-)^2}+\frac{1}{(\gamma_{j_0+1}^- -\gamma)^2}\right) r^{-2\gamma -2}.
		\end{align*}
		Therefore, since $\gamma> d-2$,
		\begin{align*}
			\frac{1}{R^d}\int_{|x|\geq R}\frac{|u|^2}{|x|^2}\, dx
			&=\frac{1}{R^d}\int_R^\infty\left(  \sum_{j=0}^\infty \frac{|u_j(r)|^2}{r^2}\right) \, r^{d-1}\, dr \\
			& 
			\leq \frac{1}{d-2}\left(\frac{1}{(\gamma-\gamma_{j_0}^-)^2}+\frac{1}{(\gamma_{j_0+1}^- -\gamma)^2}\right) \frac{R^{-2\gamma-2}}{2\gamma+2-d}\\
			&
			\leq \frac{1}{(d-2)^2}\left(\frac{1}{(\gamma-\gamma_{j_0}^-)^2}+\frac{1}{(\gamma_{j_0+1}^- -\gamma)^2}\right)  R^{-2\gamma-2},
		\end{align*}
		which proves \eqref{e.udecay}.
	\end{proof}

	\bibliographystyle{alpha}
	\bibliography{ref}

\begin{thebibliography}{ABGaS15}

\bibitem[ABGaS15]{ABGS}
Stan Alama, Lia Bronsard, and Bernardo Galv\~{a}o Sousa.
\newblock Weak anchoring for a two-dimensional liquid crystal.
\newblock {\em Nonlinear Anal.}, 119:74--97, 2015.

\bibitem[ABGL21]{ABGL21}
Stan Alama, Lia Bronsard, Dmitry Golovaty, and Xavier Lamy.
\newblock Saturn ring defect around a spherical particle immersed in a nematic
  liquid crystal.
\newblock {\em Calc. Var. Partial Differential Equations}, 60(6):Paper No. 225,
  50, 2021.

\bibitem[ABL16]{alamabronsardlamy16saturn}
Stan Alama, Lia Bronsard, and Xavier Lamy.
\newblock Minimizers of the {L}andau--de {G}ennes energy around a spherical
  colloid particle.
\newblock {\em Arch. Ration. Mech. Anal.}, 222(1):427--450, 2016.

\bibitem[ABL18]{alamabronsardlamy17}
Stan Alama, Lia Bronsard, and Xavier Lamy.
\newblock Spherical particle in nematic liquid crystal under an external field:
  the {S}aturn ring regime.
\newblock {\em J. Nonlinear Sci.}, 28(4):1443--1465, 2018.

\bibitem[ACS]{ACS22}
F.~Alouges, A.~Chambolle, and D.~Stantejsky.
\newblock Convergence to line and surface energies in nematic liquid crystal
  colloids with external magnetic field.
\newblock arXiv:2202.10703.

\bibitem[ACS21]{ACS21}
Fran\c{c}ois Alouges, Antonin Chambolle, and Dominik Stantejsky.
\newblock The {S}aturn ring effect in nematic liquid crystals with external
  field: effective energy and hysteresis.
\newblock {\em Arch. Ration. Mech. Anal.}, 241(3):1403--1457, 2021.

\bibitem[BCG05]{BCG05}
L.~Berlyand, D.~Cioranescu, and D.~Golovaty.
\newblock Homogenization of a {G}inzburg-{L}andau model for a nematic liquid
  crystal with inclusions.
\newblock {\em J. Math. Pures Appl. (9)}, 84(1):97--136, 2005.

\bibitem[Bd70]{brocharddegennes70}
{Brochard, F.} and {de Gennes, P.G.}
\newblock Theory of magnetic suspensions in liquid crystals.
\newblock {\em J. Phys. France}, 31(7):691--708, 1970.

\bibitem[Bet93]{bethuel93}
Fabrice Bethuel.
\newblock On the singular set of stationary harmonic maps.
\newblock {\em Manuscr. Math.}, 78(4):417--443, 1993.

\bibitem[BGO20]{BGO}
Peter Bella, Arianna Giunti, and Felix Otto.
\newblock Effective multipoles in random media.
\newblock {\em Comm. Partial Differential Equations}, 45(6):561--640, 2020.

\bibitem[BK05]{BK}
L.~Berlyand and E.~Khruslov.
\newblock Ginzburg-{L}andau model of a liquid crystal with random inclusions.
\newblock {\em J. Math. Phys.}, 46(9):095107, 15, 2005.

\bibitem[CDGP14]{CDGP}
M.~C. Calderer, A.~DeSimone, D.~Golovaty, and A.~Panchenko.
\newblock An effective model for nematic liquid crystal composites with
  ferromagnetic inclusions.
\newblock {\em SIAM J. Appl. Math.}, 74(2):237--262, 2014.

\bibitem[CZ20a]{CZ20design}
Giacomo Canevari and Arghir Zarnescu.
\newblock Design of effective bulk potentials for nematic liquid crystals via
  colloidal homogenisation.
\newblock {\em Math. Models Methods Appl. Sci.}, 30(2):309--342, 2020.

\bibitem[CZ20b]{CZ20}
Giacomo Canevari and Arghir Zarnescu.
\newblock Polydispersity and surface energy strength in nematic colloids.
\newblock {\em Math. Eng.}, 2(2):290--312, 2020.

\bibitem[DMP21]{DMP21}
Federico Dipasquale, Vincent Millot, and Adriano Pisante.
\newblock Torus-like solutions for the {Landau}-de {Gennes} model. {I}: {The}
  {Lyuksyutov} regime.
\newblock {\em Arch. Ration. Mech. Anal.}, 239(2):599--678, 2021.

\bibitem[GT01]{gilbargtrudinger}
D.~Gilbarg and N.S. Trudinger.
\newblock {\em Elliptic partial differential equations of second order}.
\newblock Classics in Mathematics. Springer-Verlag, Berlin, 2001.
\newblock Reprint of the 1998 edition.

\bibitem[HKL86]{hkl86}
Robert Hardt, David Kinderlehrer, and Fang-Hua Lin.
\newblock Existence and partial regularity of static liquid crystal
  configurations.
\newblock {\em Comm. Math. Phys.}, 105(4):547--570, 1986.

\bibitem[HKL88]{hkl88}
R.~Hardt, D.~Kinderlehrer, and F.-H. Lin.
\newblock Stable defects of minimizers of constrained variational principles.
\newblock {\em Ann. Inst. H. Poincar\'e Anal. Non Lin\'eaire}, 5(4):297--322,
  1988.

\bibitem[HKL90]{HKL90}
Robert Hardt, David Kinderlehrer, and Fang~Hau Lin.
\newblock The variety of configurations of static liquid crystals.
\newblock Variational methods, {Proc}. {Conf}., {Paris}/{Fr}. 1988, {Prog}.
  {Nonlinear} {Differ}. {Equ}. {Appl}. 4, 115-131 (1990)., 1990.

\bibitem[KRST96]{kuksenok96}
O.~V. Kuksenok, R.~W. Ruhwandl, S.~V. Shiyanovskii, and E.~M. Terentjev.
\newblock Director structure around a colloid particle suspended in a nematic
  liquid crystal.
\newblock {\em Phys. Rev. E}, 54:5198--5203, Nov 1996.

\bibitem[Lav20]{lavrentovich20}
Oleg~D. Lavrentovich.
\newblock Design of nematic liquid crystals to control microscale dynamics.
\newblock {\em Liquid Crystals Reviews}, 8(2):59--129, 2020.

\bibitem[LPCS98]{lubensky98}
T.~C. Lubensky, David Pettey, Nathan Currier, and Holger Stark.
\newblock Topological defects and interactions in nematic emulsions.
\newblock {\em Phys. Rev. E}, 57:610--625, Jan 1998.

\bibitem[Luc88]{luckhaus88}
S.~Luckhaus.
\newblock Partial {H}\"older continuity for minima of certain energies among
  maps into a {R}iemannian manifold.
\newblock {\em Indiana Univ. Math. J.}, 37(2):349--367, 1988.

\bibitem[Mu{\v{s}}19]{musevic19}
Igor Mu{\v{s}}evi{\v{c}}.
\newblock Interactions, topology and photonic properties of liquid crystal
  colloids and dispersions.
\newblock {\em Eur. Phys. J. Special Topics}, 227(17):2455--2485, 2019.

\bibitem[PR00]{pacardriviere}
Frank Pacard and Tristan Rivi\`ere.
\newblock {\em Linear and nonlinear aspects of vortices}, volume~39 of {\em
  Progress in Nonlinear Differential Equations and their Applications}.
\newblock Birkh\"{a}user Boston, Inc., Boston, MA, 2000.
\newblock The Ginzburg-Landau model.

\bibitem[RNRP96]{ramaswamy96}
Sriram Ramaswamy, Rajaram Nityananda, V.~A. Raghunathan, and Jacques Prost.
\newblock Power-law forces between particles in a nematic.
\newblock {\em Mol. Cryst. Liq. Cryst.}, 288(1):175--180, 1996.

\bibitem[Sch83]{schoen83}
Richard~M. Schoen.
\newblock Uniqueness, symmetry, and embeddedness of minimal surfaces.
\newblock {\em J. Differential Geom.}, 18(4):791--809 (1984), 1983.

\bibitem[Sch84]{schoen84}
R.~Schoen.
\newblock Analytic aspects of the harmonic map problem.
\newblock In {\em Seminar on nonlinear partial differential equations
  ({B}erkeley, {C}alif., 1983)}, volume~2 of {\em Math. Sci. Res. Inst. Publ.},
  pages 321--358. Springer, New York, 1984.

\bibitem[SU82]{schoenuhlenbeck82}
R.~Schoen and K.~Uhlenbeck.
\newblock A regularity theory for harmonic maps.
\newblock {\em J. Differential Geom.}, 17(2):307--335, 1982.

\bibitem[SW71]{Stein}
Elias~M. Stein and Guido Weiss.
\newblock {\em Introduction to {F}ourier analysis on {E}uclidean spaces}.
\newblock Princeton Mathematical Series, No. 32. Princeton University Press,
  Princeton, N.J., 1971.

\end{thebibliography}
	
	\end{document}